\documentclass[12pt]{article}
\usepackage[]{amsmath,amssymb}
\usepackage{amscd}
\usepackage{latexsym}
\usepackage{cite}
\usepackage{amsthm}

\newtheorem{definition}{Definition}[section]
\newtheorem{theorem}[definition]{Theorem}
\newtheorem{lemma}[definition]{Lemma}
\newtheorem{corollary}[definition]{Corollary}

\newtheorem{example}[definition]{Example}

\newtheorem{proposition}[definition]{Proposition}

\begin{document}
\title{\bf 
Using Catalan words and a $q$-shuffle 
algebra \\  
to describe a PBW basis for
the positive \\
part of 
$U_q({\widehat {\mathfrak{sl}}}_2)$
}
\author{
Paul Terwilliger 
}
\date{}

\maketitle
\begin{abstract}
The positive part $U^+_q$
of $U_q({\widehat {\mathfrak{sl}}}_2)$
has a presentation with
two generators $A,B$ that satisfy
the cubic $q$-Serre relations.
In 1993 I. Damiani obtained a PBW basis for 
 $U^+_q$,
 consisting of
some elements $\lbrace E_{n\delta+\alpha_0}\rbrace_{n=0}^\infty$,
$\lbrace E_{n\delta+\alpha_1}\rbrace_{n=0}^\infty$,
$\lbrace E_{n\delta}\rbrace_{n=1}^\infty$ that are
 defined recursively.
Our goal is to describe these elements in closed form.
To reach our goal,
start with the free associative algebra $\mathbb V$ on two generators $x,y$.
The standard (linear) basis for 
$\mathbb V$ consists of the words in $x,y$.
In 1995 M. Rosso introduced 
an associative algebra structure on $\mathbb V$, called a
$q$-shuffle algebra.
For $u,v\in \lbrace x,y\rbrace$ their 
$q$-shuffle product is
$u\star v = uv+q^{\langle u,v\rangle }vu$, where
$\langle u,v\rangle =2$ 
(resp. $\langle u,v\rangle =-2$) 
if $u=v$ (resp. 
 $u\not=v$).
Rosso gave an injective algebra homomorphism
$\natural$ from $U^+_q$ into the $q$-shuffle algebra
${\mathbb V}$, that sends $A\mapsto x$ and $B\mapsto y$.
We apply $\natural $ to the above PBW basis,
and express the image in the standard basis for $\mathbb V$.
This image involves words of the following type.
Define $\overline x = 1$
and $\overline y = -1$.
A word $a_1a_2\cdots a_n$
is {\it Catalan} whenever
$\overline a_1+
\overline a_2+\cdots  + 
\overline a_i$ is nonnegative for 
$1 \leq i \leq n-1$ and zero for $i=n$.
In this case $n$ is even.
For $n\geq 0$ define
\begin{align*}
C_n =
  \sum a_1a_2\cdots a_{2n}
\lbrack 1\rbrack_q
\lbrack 1+\overline a_1\rbrack_q
\lbrack 1+\overline a_1+\overline a_2\rbrack_q
\cdots
\lbrack 1+\overline a_1+\overline a_2+ \cdots +\overline a_{2n}\rbrack_q,
\end{align*}
where the sum is over all the Catalan words $a_1 a_2 \cdots a_{2n}$
in $\mathbb V$ that have length $2n$.
We show that
$\natural$ sends
$E_{n\delta+\alpha_0} \mapsto  q^{-2n}(q-q^{-1})^{2n} xC_n$
and $E_{n\delta+\alpha_1} \mapsto
 q^{-2n}(q-q^{-1})^{2n} C_ny$
 for $n\geq 0$, and
 $E_{n\delta} \mapsto  -q^{-2n}(q-q^{-1})^{2n-1} C_n$
 for $n\geq 1$.
It follows from this and earlier results of  Damiani
that $\lbrace C_n\rbrace_{n=1}^\infty$ mutually
 commute in the $q$-shuffle algebra $\mathbb V$.


 \end{abstract}

\section{Introduction}
The topic of Catalan sequences 
 appears in many textbooks on Combinatorics
\cite{bona},
\cite{bruIntro},
\cite{stanleyvol2},
\cite{vanlint}.
The textbook
\cite{stanley}
is devoted to this topic and its myriad connections 
to other branches of mathematics.
In the present paper we will encounter Catalan sequences in the theory
of quantum groups and $q$-shuffle algebras.
Before going into detail, we take a moment to establish some notation.
Recall the natural numbers
$\mathbb N = \lbrace 0,1,2,\ldots\rbrace$
and integers $\mathbb Z = \lbrace 0, \pm 1, \pm 2,\ldots \rbrace$.
Let $\mathbb F$ denote a field.
All vector spaces discussed in this paper are over $\mathbb F$.
All algebras discussed in this paper are associative,
over $\mathbb F$, and have a multiplicative identity.
Let $q$ denote a nonzero
scalar in $\mathbb F$ that is not a root of unity.
Recall the notation
\begin{align}
\label{eq:qinteger}
    \lbrack n \rbrack_q = \frac{q^n-q^{-n}}{q-q^{-1}} 
    \qquad \qquad n \in \mathbb N.
    \end{align}
\noindent
Define the algebra $U^+_q$ by generators $A, B$
and relations
\begin{align}
&A^3B-\lbrack 3\rbrack_q A^2BA+ 
\lbrack 3\rbrack_q ABA^2 -BA^3 = 0,
\label{eq:S1}
\\
&B^3A-\lbrack 3\rbrack_q B^2AB + 
\lbrack 3\rbrack_q BAB^2 -AB^3 = 0.
\label{eq:S2}
\end{align}
\noindent The algebra $U^+_q$ is called the {\it positive part of
$U_q({\widehat {\mathfrak{sl}}}_2)$}; see for example
\cite[p.~40]{hongkang} or
\cite[Corollary~3.2.6]{lusztig}.
The equations (\ref{eq:S1}), (\ref{eq:S2}) are called the
{\it $q$-Serre relations}.

\medskip

\noindent In 
\cite[p.~299]{damiani}
Damiani 
introduced some elements in $U^+_q$ denoted
  
  \begin{equation}
   \label{eqUq:PBWintro}
    \lbrace E_{n \delta + \alpha_0}\rbrace_{n=0}^\infty,
     \qquad \quad
      \lbrace E_{n \delta + \alpha_1}\rbrace_{n=0}^\infty,
       \qquad \quad
        \lbrace E_{n \delta}\rbrace_{n=1}^\infty.
	 \end{equation}
These elements are recursively defined as follows:
\begin{align}
E_{\alpha_0} = A, \qquad \qquad
E_{\alpha_1} = B, \qquad \qquad
E_{\delta} = q^{-2}BA-AB,
\label{eq:BA}
\end{align}
and for $n\geq 1$,
\begin{align}
&
E_{n \delta+\alpha_0} =
\frac{
\lbrack E_\delta, E_{(n-1)\delta+ \alpha_0} \rbrack
}
{q+q^{-1}},
\qquad \qquad
E_{n \delta+\alpha_1} =
 \frac{
 \lbrack
 E_{(n-1)\delta+ \alpha_1},
 E_\delta
 \rbrack
 }
 {q+q^{-1}},
 \label{eq:dam1intro}
 \\
 &
 \qquad \qquad
 E_{n \delta} =
 q^{-2}  E_{(n-1)\delta+\alpha_1} A
 - A E_{(n-1)\delta+\alpha_1}.
 \label{eq:dam2intro}
\end{align}
By 
\cite[p.~307]{damiani}
the elements $\lbrace E_{n\delta}\rbrace_{n=1}^\infty$ mutually
commute.
\medskip

\noindent 
In \cite[Section~5]{damiani}, Damiani showed that the elements
   (\ref{eqUq:PBWintro}) form a 
     Poincar\'e-Birkhoff-Witt (or PBW)
   basis for $U^+_q$. Here are the details.
\begin{definition}
\label{def:PBWorder}
\rm
We impose a linear order on the elements
(\ref{eqUq:PBWintro}) such that
 \begin{align}
 \label{eq:order}
 E_{\alpha_0} < E_{\delta+\alpha_0} <
  E_{2\delta+\alpha_0}
  < \cdots
  <
   E_{\delta} < E_{2\delta}
    < E_{3\delta}
 < \cdots
  <
   E_{2\delta + \alpha_1} <
     E_{\delta + \alpha_1} < E_{\alpha_1}.
     \end{align}
     \end{definition}
     \begin{proposition}
    \label{prop:PBWbasis} 
    {\rm (See \cite[p.~308]{damiani}.)}
     The vector space $U^+_q$ has a linear basis consisting of
    the products $x_1x_2\cdots x_n$
     $(n\in \mathbb N)$
     of elements in
    {\rm (\ref{eq:order})} such that
     $x_1 \leq x_2 \leq \cdots \leq x_n$.
     We interpret the empty product as the multiplicative identity in
     $U^{+}_q$.
   \end{proposition}

\noindent  Let $x,y$ 
denote noncommuting indeterminates, and let $\mathbb V$ denote the free
algebra with generators $x$, $y$.
For $n \in \mathbb N$, a {\it word of length $n$} in $\mathbb V$
is a product $v_1 v_2 \cdots v_n$
such that $v_i \in \lbrace x, y\rbrace$
for $1 \leq i \leq n$. We interpret the word of length zero
 to be the multiplicative
 identity in $\mathbb V$; this word is called {\it trivial} and denoted by $1$.
 The vector space $\mathbb V$ has a basis consisting
 of its words; this basis is called {\it standard}.
\medskip

\noindent We just defined the free algebra $\mathbb V$. There is
another algebra structure on $\mathbb V$,
called the $q$-shuffle algebra.
This algebra was introduced by Rosso
\cite{rosso1, rosso} and described further by Green
\cite{green}. We will adopt the approach of 
\cite{green}, which is well suited to our purpose.
The $q$-shuffle product 
is denoted by $\star$. To describe this product, we start
with some special cases.
We have $1 \star v = v \star 1 = v$ for $v \in \mathbb V$.
 For
 $u \in \lbrace x, y\rbrace$ and
  a nontrivial word $v= v_1v_2\cdots v_n$ in $\mathbb V$,
  \begin{align}
  \label{eq:Xcv}
  &u \star v =
  \sum_{i=0}^n v_1 \cdots v_{i} u v_{i+1} \cdots v_n
  q^{
  \langle v_1, u\rangle+
  \langle v_2, u\rangle+
  \cdots + \langle v_{i}, u\rangle},
  \\
  &v \star u = \sum_{i=0}^n v_1 \cdots v_{i} u v_{i+1} \cdots v_n
  q^{
  \langle  v_{n},u\rangle
  +
  \langle v_{n-1},u\rangle
  +
  \cdots
  +
  \langle v_{i+1},u\rangle
  },
  \label{eq:vcX}
  \end{align}
where
\bigskip

\centerline{
\begin{tabular}[t]{c|cc}
$\langle\,,\,\rangle$ & $x$ & $y$
   \\  \hline
   $x$ &
   $2$ & $-2$
     \\
     $y$ &
      $-2$ & $2$
         \\
	      \end{tabular}
	      }
	      \medskip

\noindent
   For example
   \begin{align*}
    &x \star y = xy+ q^{-2}yx,
 \qquad \qquad \quad
   y \star x = yx + q^{-2}xy,
\\ 
 &x \star x = (1+q^2)xx, 
 \qquad \qquad \quad
    y\star y = (1+q^2)yy,
   \\
   & x\star (yyy)= xyyy+ q^{-2} yxyy+ q^{-4} yyxy+q^{-6}yyyx,
    \\
   & (xyx) \star y = xyxy +
                (1+q^{-2})xyyx +
		  q^{-2} yxyx.
   \end{align*}
\noindent For nontrivial words $u=u_1u_2\cdots u_r$
and $v=v_1v_2\cdots v_s$ in $\mathbb V$,
\begin{align}
\label{eq:uvcirc}
&u \star v  = u_1\bigl((u_2\cdots u_r) \star v\bigr)
+ v_1\bigl(u \star (v_2 \cdots v_s)\bigr)
q^{
\langle u_1, v_1\rangle +
\langle u_2, v_1\rangle +
\cdots
+
\langle u_r, v_1\rangle},
\\
\label{eq:uvcirc2}
&u\star v =
\bigl(u \star (v_1 \cdots v_{s-1})\bigr)v_s +
\bigl((u_1 \cdots u_{r-1}) \star v\bigr)u_r
q^{
\langle u_r, v_1\rangle +
\langle u_r, v_2\rangle + \cdots +
\langle u_r, v_s\rangle
}.
\end{align}
 For example
   \begin{align*}
   (xx)\star (yyy) =& 
            xxyyy  +
            q^{-2}xyxyy + 
            q^{-4}xyyxy +
            q^{-6}xyyyx +
	     q^{-4}yxxyy 
             \\
	    &+ q^{-6}yxyxy +
             q^{-8}yxyyx +
	     q^{-8}yyxxy +
             q^{-10}yyxyx+
	     q^{-12}yyyxx,
       \\
       (xy)\star (xxyy) = & xyxxyy+ xxyyxy+ \lbrack 2 \rbrack_q^2 xxyxyy+
	\lbrack 3 \rbrack_q^2 xxxyyy.
   \end{align*}
With some work
(or by
\cite[p.~10]{green})
one
obtains
\begin{align}
&
x \star x \star x \star y -
\lbrack 3 \rbrack_q
x \star x\star y \star x +
\lbrack 3 \rbrack_q
x \star y \star x \star x -
y \star x \star x \star x  = 0,
\label{eq:qsc1}
\\
&
y \star y \star y \star x -
\lbrack 3 \rbrack_q
y \star y \star x \star y +
\lbrack 3 \rbrack_q
y \star x \star y \star y -
x \star y \star y \star y  = 0.
\label{eq:qsc2}
\end{align}
So in the $q$-shuffle algebra $\mathbb V$ the elements
$x,y$ satisfy the 
$q$-Serre relations.
Consequently there exists an algebra homomorphism
$\natural$ from $U^+_q$ to the $q$-shuffle algebra $\mathbb V$,
that sends $A\mapsto x$ and $B\mapsto y$.
The map $\natural$ is injective by
  \cite[Theorem~15]{rosso}. See 
\cite[p.~696]{leclerc} for more information about $\natural$.
\medskip

\noindent We now state our goal for the paper.
  We will apply the map $\natural $ to each element in 
   (\ref{eqUq:PBWintro}), and express the image in the standard basis
   for $\mathbb V$. As we will see, the coefficients have
   an attractive closed form. 
   We give our main theorem after a few comments. 

\begin{definition}
\label{def:cat1} \rm
Define $\overline x = 1$
and $\overline y = -1$.
Pick an integer $n\geq 0$ and consider a word $w=a_1a_2\cdots a_n$
in $\mathbb V$. 
The word $w$ is called {\it balanced} whenever
$\overline a_1+
\overline a_2+\cdots  + 
\overline a_n=0$; in this case $n$ is even.
The word 
$w$ is said to be {\it Catalan} whenever it
is balanced and 
$\overline a_1+
\overline a_2+\cdots +
\overline a_i\geq 0$ for $1 \leq i \leq n$.
\end{definition}

\begin{example}
\label{ex:CatEx} 
For $0\leq n \leq 3$ we give the Catalan words of length $2n$.
\bigskip

\centerline{
\begin{tabular}[t]{c|c}
   $n$  & {\rm Catalan words of length $2n$} 
   \\
   \hline
 $ 0 $  &  $1$
 \\
 $ 1 $  &  $xy$
 \\
 $ 2 $  &  $xyxy, \quad xxyy$
 \\
 $ 3 $  & 
 $xyxyxy,
 \quad xxyyxy,
 \quad xyxxyy,
 \quad xxyxyy,
 \quad xxxyyy$
   \end{tabular}}
\end{example}

\begin{definition}
\label{def:CnIntro}
\rm For $n\geq 0$ define
\begin{align}
C_n = 
&  \sum a_1a_2\cdots a_{2n} 
\lbrack 1\rbrack_q
\lbrack 1+\overline a_1\rbrack_q
\lbrack 1+\overline a_1+\overline a_2\rbrack_q
\cdots 
\lbrack 1+\overline a_1+\overline a_2+ \cdots +\overline a_{2n}\rbrack_q,
\label{eq:cdefIntro}
\end{align}
where the sum is over all the Catalan words $a_1 a_2 \cdots a_{2n}$
in $\mathbb V$ that have length $2n$.
We call $C_n$ the {\it $n^{\rm th}$ Catalan element} in $\mathbb V$.
\end{definition}

\begin{example} 
\label{ex:CnIntro} We have
\begin{align*}
&\qquad \qquad C_0 = 1,
\qquad \qquad
 C_1 = \lbrack 2 \rbrack_q  xy,
\qquad \qquad
 C_2 = \lbrack 2\rbrack^2_q xyxy+ 
\lbrack 3 \rbrack_q \lbrack 2 \rbrack^2_q 
 xxyy,
\\
&
C_3 = 
\lbrack 2 \rbrack^3_q
      xyxyxy + 
       \lbrack 3\rbrack_q
\lbrack 2 \rbrack^3_q
      xxyyxy
     +
     \lbrack 3 \rbrack_q
\lbrack 2 \rbrack^3_q
     xyxxyy
     +
     \lbrack 3 \rbrack^2_q
\lbrack 2 \rbrack^3_q
     xxyxyy
     +
\lbrack 4 \rbrack_q
\lbrack 3 \rbrack^2_q 
\lbrack 2 \rbrack^2_q 
     xxxyyy.
\end{align*}
\end{example}

\noindent The following is our main result.

\begin{theorem}
\label{thm:mainres} 
The map $\natural$ sends
\begin{align}
E_{n\delta+\alpha_0} \mapsto  q^{-2n}(q-q^{-1})^{2n} xC_n,
\qquad  \qquad 
E_{n\delta+\alpha_1} \mapsto
 q^{-2n}(q-q^{-1})^{2n} C_ny 
\label{eq:main1}
\end{align}
for $n\geq 0$, and
\begin{align}
E_{n\delta} \mapsto  -q^{-2n}(q-q^{-1})^{2n-1} C_n
\label{eq:main2}
\end{align}
for $n\geq 1$.
We emphasize that in 
{\rm (\ref{eq:main1})}, 
the notations $xC_n$ and $C_ny$ refer to the concatenation product.
\end{theorem}

\noindent In the algebra $U^+_q$,
the PBW basis elements
   (\ref{eqUq:PBWintro}) are known to satisfy certain relations
   \cite[Section~4]{damiani}.
Applying
$\natural $ to these relations 
and using
Theorem
\ref{thm:mainres},  we obtain some relations involving
the Catalan elements.
In Section 3 we will discuss these relations
in detail, and for now give one example.
Recall from below 
 (\ref{eq:dam2intro}) 
that $\lbrace E_{n\delta}\rbrace_{n=1}^\infty$ mutually commute.

\begin{corollary}
\label{cor:com}
For $i,j \in \mathbb N$,
\begin{equation}
C_i \star C_j = C_j \star C_i.
\end{equation}
\end{corollary}

\noindent Let $\lbrace \lambda_n \rbrace_{n=1}^\infty$
denote mutually commuting indeterminates. Let
$\mathbb F \lbrack \lambda_1, \lambda_2,\ldots \rbrack$
denote the algebra consisting of the polynomials in
$\lbrace \lambda_n \rbrace_{n=1}^\infty$ that have all coefficients
in $\mathbb F$.
Proposition
    \ref{prop:PBWbasis} and Corollary
\ref{cor:com} imply the following.

\begin{corollary} Let $C$ denote the subalgebra of
the $q$-shuffle algebra $\mathbb V$ that is generated by
$\lbrace C_n \rbrace_{n=1}^\infty$.
Then there exists an algebra
isomorphism 
$\mathbb F \lbrack \lambda_1, \lambda_2,\ldots \rbrack \to C$ 
that sends $\lambda_n \mapsto C_n$ for $n\geq 1$.
\end{corollary}

\section{The proof of Theorem
\ref{thm:mainres} }

\noindent In this section we prove Theorem
\ref{thm:mainres}. Before starting the main argument, we
introduce a map that will simplify some of our proofs.
Recall that for any algebra $\mathcal A$, 
an {\it antiautomorphism}
of $\mathcal A$ is an $\mathbb F$-linear bijection 
$\gamma : \mathcal A \to \mathcal A$ such that
$(ab)^\gamma = b^\gamma a^\gamma$ for all $a,b \in \mathcal A$.
By the form of the $q$-Serre relations
(\ref{eq:S1}),
(\ref{eq:S2}) we see that  
there exists an
antiautomorphism $\zeta$ of $U^+_q$ that swaps $A$, $B$.
Consider what $\zeta$ does to
 the PBW basis  (\ref{eqUq:PBWintro}).
By \cite[p.~307]{damiani}, the following holds in $U_q^+$ for $n\geq 1$:
\begin{align}
E_{n \delta} =
 q^{-2} B  E_{(n-1)\delta+\alpha_0}
 - E_{(n-1)\delta+\alpha_0}B.
 \label{eq:dam3intro}
 \end{align}
\noindent By this and
(\ref{eq:BA})--(\ref{eq:dam2intro}) we see that 
the map $\zeta$ fixes
$E_{n \delta}$ for $n\geq 1$, and
swaps 
$E_{n \delta + \alpha_0}$,
$E_{n \delta + \alpha_1}$ for $n\in \mathbb N$.
\medskip

\noindent We have been discussing the map $\zeta$
on $U^+_q$. There is an analogous map on
$\mathbb V$, that we will also call $\zeta$.
This $\zeta$ is the antiautomorphism
of the free algebra $\mathbb V$ that swaps $x, y$.
For example, $\zeta$ sends
\begin{align*}
xxyxxy\leftrightarrow xyyxyy, \qquad \qquad
xxyyxx \leftrightarrow yyxxyy, \qquad \qquad 
xxxyyy \leftrightarrow xxxyyy.
\end{align*}
The map $\zeta$ on $\mathbb V$ permutes each of the following sets:
(i) the words in $\mathbb V$;
(ii) the balanced words in $\mathbb V$;
(iii) the Catalan words in $\mathbb V$.
Moreover $\zeta$ fixes $C_n$ for $n\in \mathbb N$.
 By  (\ref{eq:Xcv})--(\ref{eq:uvcirc2})
the map $\zeta$ on $\mathbb V$ is an antiautomorphism
of the $q$-shuffle algebra $\mathbb V$.
Recall the map $\natural : U^+_q \to \mathbb V$ from  below
(\ref{eq:qsc2}).  By construction the following diagram commutes:

\begin{equation*}
\begin{CD}
U^+_q @>\natural  >>
               {\mathbb V}
              \\
         @V \zeta VV                   @VV \zeta V \\
         U^+_q @>>\natural >
                                  {\mathbb  V}
                        \end{CD}
   \end{equation*}

\noindent Returning to Theorem 
\ref{thm:mainres}, we now discuss our proof strategy.
One routinely checks that
(\ref{eq:main1}) holds for $n=0$ and
(\ref{eq:main2}) holds for $n=1$.
The PBW basis elements
   (\ref{eqUq:PBWintro}) satisfy the recurrence
(\ref{eq:dam1intro}),
(\ref{eq:dam2intro}),
 (\ref{eq:dam3intro}).
The candidate images of 
   (\ref{eqUq:PBWintro}) under $\natural$ are given in
(\ref{eq:main1}), (\ref{eq:main2}).
We will show that these candidate images
satisfy a recurrence analogous to
(\ref{eq:dam1intro}),
(\ref{eq:dam2intro}),
 (\ref{eq:dam3intro}),
 which looks as follows
in terms of the Catalan elements.
We will show that for $n\geq 1$,
\begin{align}
 xC_n &= \frac{ (xC_{n-1})\star (xy)- (xy) \star (x C_{n-1})}{q-q^{-1}},
\label{eq:step2}
\\
 C_ny &= \frac{ 
(xy) \star ( C_{n-1}y)
-
(C_{n-1} y)\star (xy)
}{q-q^{-1}},
\label{eq:step3}
\\
q^{-1} C_n &= \frac{q x \star (C_{n-1}y)- q^{-1} (C_{n-1}y) \star x}{q-q^{-1}},
\label{eq:step4}
\\
q^{-1} C_n &= \frac{
q  (xC_{n-1})\star y
- 
q^{-1} y \star(x C_{n-1}) 
}{q-q^{-1}}.
\label{eq:step5}
\end{align}
\noindent We will show these in order
(\ref{eq:step4}),
(\ref{eq:step5}),
(\ref{eq:step2}),
(\ref{eq:step3}).
\medskip

\noindent We just gave our proof strategy. We now proceed with
the main argument.
\begin{lemma} 
\label{lem:bal}
For a balanced word $v=a_1a_2\cdots a_m$,
\begin{align}
\frac{q x \star (vy) - q^{-1} (vy)\star x}{q-q^{-1}} 
= q^{-1} \sum_{i=0}^m 
a_1 \cdots a_{i} x a_{i+1} \cdots a_m y \lbrack 2+ 2\overline a_1 +
2 \overline a_2 + \cdots + 2 \overline a_{i} \rbrack_q.
\label{eq:exp1}
\end{align}
\end{lemma}
\begin{proof} To verify 
(\ref{eq:exp1}),
 expand the left-hand side
using 
  (\ref{eq:Xcv}),
  (\ref{eq:vcX}) and evaluate the result
using
(\ref{eq:qinteger}).
\end{proof}

\noindent 
For notational convenience, we endow the vector space
$\mathbb V$ with a certain bilinear form.
There exists a unique bilinear form 
$( \,,\,)$: $\mathbb V \times \mathbb V \to \mathbb F$ with 
respect to which the standard basis is orthonormal.
The bilinear form
$( \,,\,)$ is symmetric and nondegenerate.
For $v \in \mathbb V$ we have 
\begin{align}
v = \sum_{w} w (w, v),
\label{eq:vexpand}
\end{align}
where the sum is over all the words $w$ in $\mathbb V$.
Of course, in this sum the coefficient
$( w, v)$ is nonzero for only finitely many
$w$.

\begin{lemma}
\label{lem:balw}
For a balanced word $v=a_1a_2\cdots a_m$ and any word $w$,
\begin{align}
\Bigl(
\frac{q x \star (vy) - q^{-1} (vy)\star x}{q-q^{-1}} 
,w \Bigr)
= 
q^{-1} \sum_{i} \lbrack 2+ 2\overline a_1 +  2  \overline a_2 + \cdots + 2 \overline a_{i} \rbrack_q,
\label{eq:wvinner}
\end{align}
where the sum is over the integers $i$ $(0 \leq i \leq m)$ such that
$w=a_1a_2\cdots a_i x a_{i+1} \cdots a_m y$.
\end{lemma}
\begin{proof} In equation
(\ref{eq:exp1}), take the inner product of each side with $w$.
\end{proof}

\begin{lemma}
\label{lem:com}
Referring to Lemma
\ref{lem:balw}, assume that $v$ is Catalan and $w$ is not.
Then in 
{\rm (\ref{eq:wvinner})} each side is zero.
\end{lemma}
\begin{proof} The word $a_1a_2\cdots a_i x a_{i+1} \cdots a_my$ 
is Catalan for $0 \leq i \leq m$.
\end{proof}

\begin{definition}
\label{def:Catset}
\rm  For $n\in \mathbb N$ let
${\rm Cat}_n$ denote the set of Catalan words in $\mathbb V$ 
that have length $2n$. 
\end{definition}

\begin{definition}
\label{def:Catcoef}
\rm
For a Catalan word $w = a_1a_2\cdots a_{2n}$ define
\begin{align}
C(w) = 
\lbrack 1 \rbrack_q
\lbrack 1+\overline a_1\rbrack_q
\lbrack 1+\overline a_1+\overline a_2\rbrack_q
\cdots 
\lbrack 1+\overline a_1+\overline a_2+ \cdots + \overline a_{2n}\rbrack_q.
\label{eq:Cform}
\end{align}
\end{definition}

\noindent 
Using 
Definitions
\ref{def:Catset},
\ref{def:Catcoef} we can rewrite
(\ref{eq:cdefIntro}) as follows:
\begin{align}
\label{lem:short}
C_n = \sum_{w\in {\rm Cat}_n} w C(w)
\qquad \qquad n\in \mathbb N.
\end{align}

\begin{definition}
\label{def:elevation}
\rm
Let $w=a_1a_2\cdots a_{2n}$ denote a Catalan
word. Its {\it elevation sequence} is
$(e_0,e_1,\ldots, e_{2n})$ with
$e_i = 
\overline a_1+
\overline a_2+\cdots +
\overline a_i $  $(0 \leq i \leq 2n)$.
\end{definition}

\begin{example}\rm We display each Catalan word
of length 6, along with its elevation sequence. 
\bigskip

\centerline{
\begin{tabular}[t]{cc}
   Catalan word & elevation sequence 
   \\
   \hline
 $ xyxyxy $  &  $(0,1,0,1,0,1,0)$
 \\
 $ xyxxyy $  &  $(0,1,0,1,2,1,0)$
 \\
 $ xxyyxy $  &  $(0,1,2,1,0,1,0)$
 \\
 $ xxyxyy $  &  $(0,1,2,1,2,1,0)$
 \\
 $ xxxyyy $  &  $(0,1,2,3,2,1,0)$
   \end{tabular}}
\end{example}

\noindent 
Our next goal is the following.
For each Catalan word $w$ we define
a sequence of integers called the profile of $w$.
\begin{definition}
\label{def:profile}
\rm
Let $w$ denote a Catalan word. The {\it profile} of $w$
is obtained from its
 elevation sequence
$(e_0,e_1,\ldots, e_{2n})$ by deleting all the
$e_i$ such that $1 \leq i \leq 2n-1$  and
$e_{i+1}-e_i$,
$e_i-e_{i-1}$
have the same sign.
\end{definition}

\begin{example}\rm We display each Catalan word
of length 6, along with its profile.
\bigskip

\centerline{
\begin{tabular}[t]{cc}
   Catalan word & profile 
   \\
   \hline
 $ xyxyxy $  &  $(0,1,0,1,0,1,0)$
 \\
 $ xyxxyy $  &  $(0,1,0,2,0)$
 \\
 $ xxyyxy $  &  $(0,2,0,1,0)$
 \\
 $ xxyxyy $  &  $(0,2,1,2,0)$
 \\
 $ xxxyyy $  &  $(0,3,0)$
   \end{tabular}}
\end{example}

\noindent The profile notation will play an important role
in our main calculations, so let us consider it from
several points of view. We will do this over the next three lemmas.
For these lemmas the proof is routine, and omitted.
\medskip

\noindent The elevation sequence of a Catalan word is determined
by its profile in the following way.
\begin{lemma}
\label{lem:elevation}
The elevation sequence
of a Catalan word with profile
$(\ell_0, h_1, \ell_1, h_2,  \ldots, h_{r},\ell_{r})$ is
given below:
\begin{align*}
&(0,1,2,\ldots,
h_1-1,h_1, h_1-1,\ldots,
\ell_1+1,\ell_1, \ell_1+1, \ldots, 
h_2-1,h_2, h_2-1,\ldots
\\
&
\qquad \qquad \qquad 
\qquad \qquad \qquad 
\ldots,
\ell_2+1,\ell_2, \ell_2+1, \ldots, 
h_r-1,h_r, h_r-1,\ldots 2,1,0).
\end{align*}
\end{lemma}

\noindent The length of a Catalan word is determined by its
profile in the following way.

\begin{lemma} For a Catalan word $a_1a_2\cdots a_{2n}$ with
profile
$(\ell_0, h_1, \ell_1, h_2,  \ldots, h_{r},\ell_{r})$, 
\begin{align*}
n = \sum_{i=1}^r (h_i - \ell_i).
\end{align*}
\end{lemma}

\noindent Next we clarify which sequences
are the profile of a Catalan word.   

\begin{lemma}
\label{lem:profC}
Given $r \in \mathbb N$ and a sequence of integers
\begin{align}
(\ell_0,h_1, \ell_1, h_2, \ldots, h_{r},\ell_{r}).
\label{ex:hell}
\end{align}
There exists a Catalan word $w$ with profile 
{\rm (\ref{ex:hell})} if and only if the following 
{\rm (i)--(v)} hold:
\begin{enumerate}
\item[\rm (i)]
$\ell_0 = 0$;
\item[\rm (ii)]
 $\ell_i \geq 0$ $(1 \leq i \leq r-1)$;
\item[\rm (iii)]
$\ell_r = 0$;
\item[\rm (iv)]
$\ell_{i-1} < h_{i}$ $(1 \leq i \leq r)$;
\item[\rm (v)]
$h_i > \ell_{i}$ $(1 \leq i \leq r)$.
\end{enumerate}
In this case
\begin{align*}
w=x^{h_1} y^{h_1-\ell_1} x^{h_2-\ell_1} y^{h_2-\ell_2} \cdots 
x^{h_r-\ell_{r-1}}y^{h_r}.
\end{align*}
(In the above line the exponents are with respect to the
concatentation product).
\end{lemma}

\begin{definition}\rm By a {\it Catalan profile} we mean
a sequence of integers 
$(\ell_0,h_1, \ell_1, h_2, \ldots, h_{r},\ell_{r})$
that satisfies the conditions (i)--(v) in Lemma
\ref{lem:profC}.
\end{definition}

\noindent Recall the notation
\begin{align*}
\lbrack n \rbrack^{!}_q = 
\lbrack n \rbrack_q  
\lbrack n-1 \rbrack_q  
\cdots 
\lbrack 2 \rbrack_q  
\lbrack 1 \rbrack_q
\qquad \qquad n \in \mathbb N.
\end{align*}
\noindent We interpret $\lbrack 0 \rbrack^{!}_q = 1$.
\medskip

\noindent Let $w$ denote a Catalan word. Next we give the
coefficient $C(w)$ in terms of the profile of $w$.

\begin{lemma}
\label{lem:Cform}
For a Catalan word $w$ with 
profile $(\ell_0,h_1, \ell_1, h_2, \ldots, h_{r}, \ell_{r})$,
\begin{align}
C(w) 
 =
 \frac{
\lbrack h_1 \rbrack^!_q 
\lbrack h_2 \rbrack^!_q 
\cdots 
\lbrack h_r \rbrack^!_q 
}{
\lbrack \ell_0 \rbrack^!_q 
\lbrack \ell_1 \rbrack^!_q 
\cdots 
\lbrack \ell_{r} \rbrack^!_q 
}
\times
\frac{
\lbrack  h_1+1 \rbrack^!_q 
\lbrack h_2+1 \rbrack^!_q 
\cdots 
\lbrack h_r+1 \rbrack^!_q 
}{
\lbrack \ell_0+1 \rbrack^!_q 
\lbrack \ell_1+1 \rbrack^!_q 
\cdots 
\lbrack \ell_{r}+1 \rbrack^!_q 
}.
\label{eq:Cw}
\end{align}
\end{lemma}
\begin{proof} 
Use (\ref{eq:Cform}) and
Lemma
\ref{lem:elevation}, keeping in mind that
$\overline x = 1$ and $\overline y=-1$.
\end{proof}

\noindent The next definition is for notational convenience; it
is justified by Lemma
\ref{lem:Cform}.

\begin{definition}
\label{def:Cform}
\rm For $r \in \mathbb N$ and 
any sequence of natural numbers
$(\ell_0, h_1, \ell_1, h_2, \ldots, h_r, \ell_r)$
define
\begin{align*}
C(\ell_0, h_1, \ell_1, h_2, \ldots, h_r, \ell_r)=
 \frac{
\lbrack h_1 \rbrack^!_q 
\lbrack h_2 \rbrack^!_q 
\cdots 
\lbrack h_r \rbrack^!_q 
}{
\lbrack \ell_0 \rbrack^!_q 
\lbrack \ell_1 \rbrack^!_q 
\cdots 
\lbrack \ell_{r} \rbrack^!_q 
}
\times
\frac{
\lbrack  h_1+1 \rbrack^!_q 
\lbrack h_2+1 \rbrack^!_q 
\cdots 
\lbrack h_r+1 \rbrack^!_q 
}{
\lbrack \ell_0+1 \rbrack^!_q 
\lbrack \ell_1+1 \rbrack^!_q 
\cdots 
\lbrack \ell_{r}+1 \rbrack^!_q 
}.
\end{align*}
\end{definition}

\noindent We mention two identities for later use.
For $n \in \mathbb N$,
\begin{align}
\sum_{t=1}^n 
\lbrack 2t \rbrack_q
= \lbrack n \rbrack_q \lbrack n+1 \rbrack_q.
\label{lem:geom}
\end{align}
For $r,s \in \mathbb N$ with $r<s$,
\begin{align}
\sum_{t=r+1}^s 
\lbrack 2t \rbrack_q
=
\lbrack s \rbrack_q \lbrack s+1 \rbrack_q-
\lbrack r \rbrack_q \lbrack r+1 \rbrack_q.
\label{lem:geom2}
\end{align}

\begin{proposition}
\label{prop:tech1} For a Catalan profile
$(\ell_0,h_1, \ell_1, h_2, \ldots,h_{r}, \ell_{r})$ with $r\geq 1$,
\begin{align*}
&
C(\ell_0,h_1, \ell_1, h_2,  \ldots,h_{r}, \ell_{r})
\\
&= 
\sum_{i=\xi}^{r-1}
C(\ell_0,h_1, \ell_1, \ldots, h_i, \ell_i, h_{i+1}-1, \ell_{i+1}-1,
\ldots, \ell_{r-1}-1,h_r-1,\ell_r)
\times \Biggl( \sum_{t={\ell_i+1}}^{h_{i+1}} \lbrack 2t\rbrack_q \Biggr).
\end{align*}
where 
\begin{align*}
\xi = {\rm max} \lbrace i | 0 \leq i \leq r-1, \;\; \ell_i = 0\rbrace.
\end{align*}
\end{proposition}
\begin{proof} Evaluate the left-hand side using
Definition \ref{def:Cform}.
Evaluate the right-hand side using
Definition \ref{def:Cform}
and (\ref{lem:geom}),
(\ref{lem:geom2}). After some straightforward algebraic
manipulation the two sides are found to be equal.
\end{proof}

\begin{corollary}
\label{cor:aver}
For $n\geq 1$ and
 $w \in {\rm Cat}_n$ we have 
\begin{align}
\label{eq:needit}
C(w) = q \sum_{v \in {\rm Cat}_{n-1}} C(v) 
 \Bigl(
\frac{q x \star (vy) - q^{-1} (vy)\star x}{q-q^{-1}} 
,w \Bigr).
\end{align}
\end{corollary}
\begin{proof} This is a reformulation of
Proposition
\ref{prop:tech1}, using
Lemma \ref{lem:balw}.
\end{proof}

\begin{proposition}
\label{prop:ww}
For $n\geq 1$,
\begin{align}
q^{-1} C_n = \frac{qx \star (C_{n-1}y)-q^{-1} (C_{n-1}y)\star x}{q-q^{-1}}.
\label{eq:step4a}
\end{align}
\end{proposition}
\begin{proof} 
For all words $w$  in $\mathbb V$, we show that each side of
(\ref{eq:step4a}) 
has the same inner product with $w$.
For the left-hand side of
(\ref{eq:step4a}),
 this inner product is $q^{-1} C(w)$ if
$w \in {\rm Cat}_n$, and 0 if $w \not\in {\rm Cat}_n$.
For the right-hand side of
(\ref{eq:step4a}),  this inner product is
\begin{align}
\sum_{v \in {\rm Cat}_{n-1}} C(v) 
 \Bigl(
\frac{q x \star (vy) - q^{-1} (vy)\star x}{q-q^{-1}} 
,w \Bigr).
\label{eq:almost}
\end{align}
First assume that $w \in {\rm Cat}_n$.
Then
the scalar (\ref{eq:almost}) is equal to
$q^{-1}C(w)$, in view of Corollary
\ref{cor:aver}. Next assume that $w \not\in {\rm Cat}_n$.
Then the scalar (\ref{eq:almost}) is equal to 0 by construction and
Lemma
\ref{lem:com}.
In any case, each side of (\ref{eq:step4a}) 
has the same inner product with $w$.
The result follows.
\end{proof}

\begin{proposition}
\label{prop:wwdual}
For $n\geq 1$,
\begin{align}
q^{-1} C_n = \frac{
q  (xC_{n-1})\star y
- 
q^{-1} y \star (x C_{n-1}) 
}{q-q^{-1}}.
\label{eq:step5a}
\end{align}
\end{proposition}
\begin{proof}  Apply the antiautomorphism $\zeta$ to each side of
(\ref{eq:step4a}).
\end{proof}

\begin{proposition}
\label{prop:last}
For $n\geq 1$,
\begin{align}
& xC_n = 
\frac{ (xC_{n-1})\star (xy)- (xy) \star (x C_{n-1})}{q-q^{-1}}.
\label{eq:step2f}
\end{align}
\end{proposition}
\begin{proof} Using 
(\ref{eq:uvcirc})
we obtain
\begin{align}
(xC_{n-1})\star (xy) &= 
x \bigl(C_{n-1}\star (xy)\bigr)
+ x\bigl((x C_{n-1}) \star y\bigr) q^2,
\label{eq:26a}
\\
(xy)\star (xC_{n-1}) &= 
x\bigl(y \star (x C_{n-1})\bigr)
+
x \bigl((xy)\star C_{n-1}\bigr).
\label{eq:26b}
\end{align}
\noindent
We claim that
\begin{align}
(xy)\star C_{n-1} = C_{n-1}\star (xy).
\label{eq:ind}
\end{align}
If $n=1$ then 
(\ref{eq:ind}) holds since $C_0=1$, so assume $n\geq 2$.
Below
 (\ref{eq:dam2intro}) 
we mentioned that
$E_{\delta}$,
$E_{2\delta}$,
$E_{3\delta},\ldots $ mutually commute.
So by (\ref{eq:main2}) and induction on $n$,
the elements $C_1, C_2, \ldots, C_{n-1}$ mutually commute
with respect to the $q$-shuffle product.
In particular $C_1, C_{n-1}$ commute with respect to the
$q$-shuffle product. The element $C_1$ 
is a nonzero scalar multiple of $xy$, so
(\ref{eq:ind}) holds and the claim is proved.
Using
(\ref{eq:step5})
 and
(\ref{eq:26a}),
(\ref{eq:26b}), 
 (\ref{eq:ind})
we obtain
\begin{align*}
x C_n & = 
 q x \, \frac{q 
 (xC_{n-1})\star y-q^{-1} y\star (xC_{n-1})}{q-q^{-1}}
\\
&
=
\frac{ (xC_{n-1})\star (xy)- (xy) \star (x C_{n-1})}{q-q^{-1}}.
\end{align*}
\end{proof}

\begin{proposition}
For $n \geq 1$,
\begin{align}
 C_ny = \frac{ 
(xy) \star ( C_{n-1}y)
-
(C_{n-1} y)\star (xy)
}{q-q^{-1}},
\label{eq:step3done}
\end{align}
\end{proposition}
\begin{proof} 
Apply the antiautomorphism $\zeta$ to each side of
(\ref{eq:step2f}).
\end{proof}

\noindent We have shown that
(\ref{eq:step2})--(\ref{eq:step5}) hold for $n\geq 1$.
Theorem 
\ref{thm:mainres} is now proven.

\section{Some relations involving the Catalan elements}

\noindent In Corollary 
\ref{cor:com} we saw that the Catalan elements
mutually commute with respect to the $q$-shuffle product.
In this section we give some more relations along this line.
First we recall some relations in the algebra $U^+_q$.

\begin{lemma}
\label{lem:dam2}
{\rm (See \cite[p.~307]{damiani}.)}
For $i,j \in \mathbb N$ the following holds in $U_q^+$:
\begin{align*}
E_{i\delta+\alpha_1}
E_{j\delta+\alpha_0}
=
q^2
E_{j\delta+\alpha_0}
E_{i\delta+\alpha_1}
+
q^2 E_{(i+j+1)\delta}.
\end{align*}
\end{lemma}

\begin{lemma}
\label{lem:com2}
{\rm (See \cite[p.~304]{damiani}.)}
For $i\geq 1$ and $j\geq 0$ the following hold in $U^+_q$:
\begin{align*}
&E_{i\delta} E_{j\delta+\alpha_0} =
E_{j\delta+ \alpha_0} E_{i\delta}
+ q^{2-2i}(q+q^{-1}) E_{(i+j)\delta+\alpha_0}
\\
& \qquad \qquad \qquad \qquad \qquad  \qquad -\;
q^2(q^2-q^{-2})\sum_{\ell=1}^{i-1}
q^{-2\ell}
E_{(j+\ell) \delta+\alpha_0}
E_{(i-\ell) \delta},
\\
&
E_{j\delta+\alpha_1}
E_{i\delta}
=
E_{i\delta}
E_{j\delta+ \alpha_1}
+ q^{2-2i}(q+q^{-1}) E_{(i+j)\delta+\alpha_1}
\\
& \qquad \qquad \qquad \qquad \qquad \qquad -\;
q^2(q^2-q^{-2})\sum_{\ell=1}^{i-1}
q^{-2\ell}
E_{(i-\ell) \delta}
E_{(j+\ell) \delta+\alpha_1}.
\end{align*}
\end{lemma}

\begin{lemma}
\label{lem:com3}
{\rm (See \cite[p.~300]{damiani}.)}
For  $i> j\geq 0$ the following hold in $U^+_q$.
\begin{enumerate}
\item[\rm (i)] Assume that $i-j=2r+1$ is odd. Then
\begin{align*}
&
E_{i\delta+\alpha_0}
E_{j\delta+\alpha_0}
=
q^{-2}
E_{j\delta+\alpha_0}
E_{i\delta+\alpha_0}
-
(q^2-q^{-2})\sum_{\ell=1}^{r}
q^{-2\ell}
E_{(j+\ell) \delta+\alpha_0}
E_{(i-\ell) \delta+\alpha_0},
\\
&
E_{j\delta+\alpha_1}
E_{i\delta+\alpha_1} =
q^{-2}
E_{i\delta+\alpha_1 }
E_{j\delta+\alpha_1 }
-
(q^2-q^{-2})\sum_{\ell=1}^{r}
q^{-2\ell}
E_{(i-\ell) \delta+\alpha_1}
E_{(j+\ell) \delta+\alpha_1}.
\end{align*}
\item[\rm (ii)] Assume that $i-j=2r$ is even. Then
\begin{align*}
E_{i\delta+\alpha_0}
E_{j\delta+\alpha_0}
 =
q^{-2}
E_{j\delta+\alpha_0}
&E_{i\delta+\alpha_0}
-
q^{j-i+1} (q-q^{-1}) E^2_{(r+j)\delta+\alpha_0}
\\
&-\;
(q^2-q^{-2})\sum_{\ell=1}^{r-1}
q^{-2\ell}
E_{(j+\ell) \delta+\alpha_0}
E_{(i-\ell) \delta+\alpha_0},
\\
E_{j\delta+\alpha_1}
E_{i\delta+\alpha_1} =
q^{-2}
E_{i\delta+\alpha_1 }
&
E_{j\delta+\alpha_1 }
-
q^{j-i+1} (q-q^{-1}) E^2_{(r+j)\delta+\alpha_1}
\\
&-\;
(q^2-q^{-2})\sum_{\ell=1}^{r-1}
q^{-2\ell}
E_{(i-\ell) \delta+\alpha_1}
E_{(j+\ell) \delta+\alpha_1}.
\end{align*}
\end{enumerate}
\end{lemma}

\noindent We mention an alternate version of
Lemma \ref{lem:com2}.

\begin{lemma}
\label{lem:altxx}
For $i,j\in \mathbb N$ the following hold in $U^+_q$:
\begin{align*}
&
E_{i\delta+\alpha_0} E_{(j+1)\delta} -
E_{(j+1)\delta} E_{i\delta+\alpha_0}
=
q^2 
E_{(i+1)\delta+\alpha_0} E_{j\delta} -
q^{-2} 
E_{j\delta} E_{(i+1)\delta+\alpha_0},
\\
&
E_{(j+1)\delta} E_{i\delta+\alpha_1}
-
E_{i\delta+\alpha_1} E_{(j+1)\delta} 
=
q^2 
E_{j\delta} E_{(i+1)\delta+\alpha_1}
-
q^{-2} 
E_{(i+1)\delta+\alpha_1} E_{j\delta}.
\end{align*}
\end{lemma}
\begin{proof}
To verify these equations,
use 
Lemma \ref{lem:com2} to
evaluate each product 
 that is out of order 
with respect to the linear order in Definition
\ref{def:PBWorder}.
\end{proof}

\noindent We mention an alternate version of
Lemma
\ref{lem:com3}.

\begin{lemma}
\label{lem:dam1} The following relations hold in $U^+_q$.
For $i \in \mathbb N$,
\begin{align}
&
q E_{(i+1)\delta+\alpha_0}
E_{i\delta+\alpha_0}
=
q^{-1}
E_{i\delta+\alpha_0}
E_{(i+1)\delta+\alpha_0},
\label{eq:spcase1}
\\
&
q E_{i\delta+\alpha_1}
E_{(i+1)\delta+\alpha_1}
=
q^{-1}
E_{(i+1)\delta+\alpha_1}
E_{i\delta+\alpha_1}.
\label{eq:spcase2}
\end{align}
\noindent For distinct $i,j \in \mathbb N$,
\begin{align*}
q
E_{(i+1)\delta+\alpha_0}
E_{j\delta+\alpha_0}
&-
q^{-1}
E_{j\delta+\alpha_0}
E_{(i+1)\delta+\alpha_0}
\\
&=\;
q^{-1}
E_{i\delta+\alpha_0}
E_{(j+1)\delta+\alpha_0}
-
q
E_{(j+1)\delta+\alpha_0}
E_{i\delta+\alpha_0},
\\
q
E_{j\delta+\alpha_1}
E_{(i+1)\delta+\alpha_1}
&-
q^{-1}
E_{(i+1)\delta+\alpha_1}
E_{j\delta+\alpha_1}
\\
&=\;
q^{-1}
E_{(j+1)\delta+\alpha_1}
E_{i\delta+\alpha_1}
-
q
E_{i\delta+\alpha_1}
E_{(j+1)\delta+\alpha_1}.
\end{align*}
\end{lemma}
\begin{proof} 
The equations (\ref{eq:spcase1}),
(\ref{eq:spcase2}) come from
Lemma \ref{lem:com3}(i) with $r=0$.
We now verify the remaining two equations in
the lemma statement.
Without loss, we many assume that $i>j$.
Under this assumption use
Lemma \ref{lem:com3} to
evaluate each product 
 that is out of order 
with respect to the linear order in Definition
\ref{def:PBWorder}.
\end{proof}

\noindent We just gave some relations in $U^+_q$. Applying
the map $\natural$ and using Theorem
\ref{thm:mainres} we routinely obtain the following results.

\begin{corollary} 
For $i,j\in \mathbb N$ the following holds in $\mathbb V$:
\begin{align*}
q^{-1}C_{i+j+1} = \frac{ q (xC_i)\star (C_jy)-q^{-1} (C_j y)\star (x C_i)}
{q-q^{-1}}.
\end{align*}
\end{corollary}

\begin{corollary} 
For
 $i,j\in \mathbb N$ the following hold in $\mathbb V$:
\begin{align*}
&\frac{(xC_i)\star C_j - C_j \star (xC_i)}{q^2-q^{-2}} = 
\sum_{\ell=1}^j q^{2-2\ell} (x C_{i+\ell})\star C_{j-\ell},
\\
&\frac{C_j \star (C_iy)-(C_i y)\star C_j}{q^2-q^{-2}} = 
\sum_{\ell=1}^{j} q^{2-2\ell} C_{j-\ell}\star (C_{i+\ell}y).
\end{align*}
\end{corollary}

\begin{corollary}
For $i>j\geq 0$ the following hold in $\mathbb V$.
\begin{enumerate}
\item[\rm (i)] Assume that $i-j=2r+1$ is odd. Then
\begin{align*}
&
\frac{
 q (x C_{i})\star (x C_j) - q^{-1} (x C_j)\star (x C_i)}{q^2-q^{-2}} =
-\sum_{\ell=1}^r q^{1-2\ell}(xC_{j+\ell})\star(xC_{i-\ell}),
\\
&
\frac{
 q ( C_{j}y)\star ( C_iy) - q^{-1} ( C_i y)\star ( C_j y)}{q^2-q^{-2}} =
-\sum_{\ell=1}^r q^{1-2\ell}(C_{i-\ell}y)\star(C_{j+\ell}y).
\end{align*}
\item[\rm (ii)]
Assume that  $i-j=2r$ is even. Then
\begin{align*}
&\frac{
 q (x C_i)\star (x C_j) - q^{-1} (x C_j)\star (x C_i)}{q^2-q^{-2}}
 + \frac{q^{j-i+2}(xC_{j+r})\star(xC_{i-r})}{q+q^{-1}} 
\\ 
 & \qquad \qquad \qquad  = -
\sum_{\ell=1}^{r-1} q^{1-2\ell}(xC_{j+\ell})\star(xC_{i-\ell}),
\\
&\frac{
 q ( C_j y)\star ( C_i y) - q^{-1} (C_i y)\star (C_j y)}{q^2-q^{-2}} 
 + \frac{q^{j-i+2}(C_{i-r}y)\star(C_{j+r}y)}{q+q^{-1}}
\\
& \qquad \qquad \qquad 
= -\sum_{\ell=1}^{r-1} q^{1-2\ell}(C_{i-\ell}y)\star(C_{j+\ell}y).
\end{align*}
\end{enumerate}
\end{corollary}

\begin{corollary} For $i,j\in \mathbb N$ the following
hold in $\mathbb V$:
\begin{align*}
&(x C_i)\star C_{j+1} - C_{j+1} \star (xC_i) =
q^2 (x C_{i+1})\star C_{j} - q^{-2} C_{j} \star (xC_{i+1}),
\\
&
C_{j+1}
\star
( C_iy)
- 
(C_iy)
\star
C_{j+1}
=
q^2
C_{j}
\star 
(C_{i+1}y)
- q^{-2}
(C_{i+1}y)
\star
C_{j}.
\end{align*}
\end{corollary}

\begin{corollary}
The following relations hold in $\mathbb V$.
For $i \in \mathbb N$,
\begin{align*}
&
q (x C_{i+1})\star (x C_i) = q^{-1} (x C_i)\star (x C_{i+1}),
\\
&
q (C_i y)\star (C_{i+1}y) = q^{-1} (C_{i+1}y)\star (C_i y).
\end{align*}
For distinct $i,j\in \mathbb N$,
\begin{align*}
&
q (x C_{i+1})\star (x C_j) - q^{-1} (x C_j)\star (x C_{i+1}) =
q^{-1} (x C_{i})\star (x C_{j+1}) - q (x C_{j+1})\star (x C_i),
\\
&
q (C_j y)\star (C_{i+1}y) - q^{-1} (C_{i+1}y)\star (C_j y) =
q^{-1} (C_{j+1} y)\star ( C_i y) - q (C_{i} y)\star ( C_{j+1} y).
\end{align*}
\end{corollary}

\bigskip

\noindent Paul Terwilliger \hfil\break
\noindent Department of Mathematics \hfil\break
\noindent University of Wisconsin \hfil\break
\noindent 480 Lincoln Drive \hfil\break
\noindent Madison, WI 53706-1388 USA \hfil\break
\noindent email: {\tt terwilli@math.wisc.edu }\hfil\break

\end{document}